\documentclass[a4paper,12pt]{article}
\usepackage{latexsym}
\usepackage{amsmath,amsthm,amssymb}
\usepackage{amsfonts}
\usepackage{eufrak}

\newtheorem{thm}{Theorem}[section]
\newtheorem{lem}[thm]{Lemma}
\newtheorem{pro}[thm]{Proposition}
\newtheorem{cor}[thm]{Corollary}

\theoremstyle{definition}

\newtheorem{exa}[thm]{Example}
\newtheorem{rem}[thm]{Remark}

\begin{document}

\begin{center}
{\Large Bounds on tail probabilities for quadratic forms in dependent sub-gaussian random variables}
\end{center}
\begin{center}
{\sc Krzysztof Zajkowski}\\
Institute of Mathematics, University of Bialystok \\ 
Ciolkowskiego 1M, 15-245 Bialystok, Poland \\ 
kryza@math.uwb.edu.pl 
\end{center}

\begin{abstract}
We show bounds on tail probabilities for quadratic forms in sub-gaussian non-necessarily independent random variables. Our main tool will be estimates of the Luxemburg norms of such forms. This will allow us  to formulate the above-mentioned bounds. As an example we give estimates of the excess loss in fixed design linear regression in dependent observations. 

\end{abstract}


  
\section{Introduction}

The most known  estimate of tail probabilities for quadratic forms is  the Hanson--Wright inequality regarding independent centered sub-gaussian random variables. A version of this inequality was first  proved in \cite{HanWri,Wri} and recently derived by Rudelson and Vershynin in \cite{RudVer}.  In this article we show estimates for quadratic forms in dependent random variables. Our main technique will be  estimates of sub-exponential norms of quadratic forms in sub-gaussian random variables.

A random variable $\xi$ is called sub-gaussian if it is dominated by Gaussian random variable. For centered $\xi$ it can be expressed by requiring that 
there is a constant $K$ such that $\mathbb{E}\exp(t\xi)\le \exp(K^2t^2/2)$ for all $t\in\mathbb{R}$ (see Kahane \cite{Kahane}). It means that its moment generating function
is majorized by the moment generating function of the centered Gaussian variable with the standard deviation $K$. The infimum of such $K$ is a norm on  the space of centered sub-gaussian random variables. It is standardly denoted by  $\tau(\cdot)$, and this space itself by $Sub(\Omega)$ on some probability space $(\Omega,\mathcal{F},\mathbb{P})$ (see Buldygin and Kozachenko \cite[Def.1.1 in Ch.1]{BulKoz}).

Non-necessarily centered sub-gaussian random variable $\xi$ can be defined by requiring that $\mathbb{E}\exp(\xi^2/K^2)\le 2$ for some $K>0$; the infimum of such $K$ is the Luxemburg norm on the Orlicz space generated by the function $\psi_2(t)=\exp(t^2)-1$. We will denote this space   by $L_{\psi_2}(\Omega)$ and the Luxemburg or $\psi_2$-norm by $\|\cdot\|_{\psi_2}$. Let us note that the norms $\tau(\cdot)$ and $\|\cdot\|_{\psi_2}$ are equivalent on the space 
$Sub(\Omega)=\{\xi\in L_{\psi_2}(\Omega):\;\mathbb{E}\xi=0\}$ (compare Vershynin \cite[Prop. 2.5.2]{Ver}).  
From now on let $C_2$ denote a universal constant such that $\tau(\xi)\le C_2\|\xi\|_{\psi_2}$ for $\xi\in Sub(\Omega)$.
A number of other equivalent definitions of norms are used in the literature, for instance by using estimates of moments of sub-gaussian random variables, but for our purposes, the above-mentioned norms will be sufficient. 

One can similarly define sub-exponential random variables, i.e. by requiring that their $\psi_1$-norm $\|\xi\|_{\psi_1}:=\inf\{K>0:\;\mathbb{E}\exp|\xi/K|\le 2\}<\infty$, where the function $\psi_1(t)=\exp(|t|)-1$. Let us note that if $\xi$ is a sub-gaussian random variable then $\xi^2$ is sub-exponential one and moreover $\|\xi^2\|_{\psi_1}=\|\xi\|_{\psi_2}^2$ (see \cite[Lem.2.7.6]{Ver} for instance).

Let us emphasize that for centered sub-exponential random variables $\xi$ there is not a global  estimate of  the moment generating function  as in sub-gaussian case. We can only formulate the following inequality 
\begin{equation}
\label{estMgf}
\mathbb{E}\exp(t\xi)\le \exp(C_1^2\|\xi\|_{\psi_1}^2t^2)
\end{equation}
for  $|t|\le 1/(C_1\|\xi\|_{\psi_1})$, where $C_1$ is an universal constant \footnote{The universal constants $C_1$, $C_2$ and, in consequence, $C_3$, $C_4$ will be the same in each occurrence and  universal constants $c$ may be different.}; compare Vershynin \cite[Prop. 2.7.1]{Ver}.

Let us recall also the Hanson--Wright inequality (see Rudelson and Vershynin \cite[Th.1.1]{RudVer}).
\begin{thm}
({\rm Hanson--Wright inequality})
Let $\xi=(\xi_1,...,\xi_n)\in \mathbb{R}^n$ be a random vector with independent coordinates  $\xi_i$ which satisfy $\mathbb{E}\xi_i=0$ and $\|\xi_i\|_{\psi_2}\le K$.
Let $A=[a_{ij}]_{i,j=1}^n$ be an $n\times n$ matrix. Then, for every $t\ge 0$, 
$$
\mathbb{P}\Big(\big|\xi^TA\xi-\mathbb{E}(\xi^TA\xi)\big|\ge t\Big)\le 
2\exp\Big(-c\min\Big\{\frac{t^2}{K^4\|A\|^2_{HS}}, \frac{t}{K^2\|A\|} \Big\}\Big),
$$
where $c$ is a universal constant, $\|A\|_{HS}=(\sum_{i,j=1}^na_{ij}^2)^{1/2}$ is the Hilbert-Schmidt norm of $A$, whereas $\|A\|=\sup_{|x|_2\le 1}|Ax|_2$ is the operator norm ($|\cdot|_2$ denotes the standard Euclidean norm in $\mathbb{R}^n$).
\end{thm}

We will need a notion of sub-gaussian random vectors. One of the ways to define such vectors is by requiring that one dimensional marginals 
$\left\langle \xi,{\bf t}\right\rangle$ are sub-gaussian for all ${\bf t}\in\mathbb{R}^n$. The sub-gaussian norm of $\xi$ is defined as 
$\|\xi\|_{\psi_2}=\sup_{{\bf t}\in S^{n-1}}\|\left\langle \xi,{\bf t}\right\rangle\|_{\psi_2}$ (see \cite[Def. 5.22]{Ver}. Using the Luxemburg norm we can rewrite it as follows
$$
\|\xi\|_{\psi_2}=\inf\Big\{K>0:\;\sup_{{\bf t}\in S^{n-1}}\;\mathbb{E}\exp\big(\left\langle \xi,{\bf t}\right\rangle^2/K^2\big)\le 2\Big\}.
$$
The space of $n$-dimensional sub-gaussian random vectors we will denote by $L_{\psi_2}^n(\Omega)$.

Similarly as in one-dimensional case we can also introduce the definition of a norm for centered sub-gaussian random vectors $\xi$ in the form
$$
\tau(\xi)=\inf\Big\{K>0:\;\forall_{{\bf t}\in\mathbb{R}^n}\;\mathbb{E}\exp\left\langle \xi,{\bf t}\right\rangle\le \exp\big(K^2|{\bf t}|_2^2/2\big)\Big\};
$$
as in the case of the norm $\|\cdot\|_{\psi_2}$ we keep the notation $\tau(\cdot)$ in the multidimensional case.
The space of centered $n$-dimensional sub-gaussian random vectors we will denote by 
$Sub^n(\Omega)=\{\xi\in L_{\psi_2}^n(\Omega):\;\mathbb{E}\xi={\bf 0}\in \mathbb{R}^n\}$.

\begin{exa}
Recall that the moment generating function of $g=(g_i)_{i=1}^n\sim\mathcal{N}({\bf 0},I_n)$ ($g_i$ are independent standard normally distributed)  equals 
$\mathbb{E}\exp\left\langle g,{\bf t}\right\rangle 
=\exp(|{\bf t}|_2^2/2)$. It means that $\tau(g)=1$. Let $A$ be $n\times n$ matrix. The random vector $Ag$ has the centered Gaussian distribution with the covariance matrix  $Cov(Ag)=AA^T$. The moment generating function of $Ag$ can be estimated as follows
$$
\mathbb{E}\exp\left\langle Ag,{\bf t}\right\rangle=\exp\Big(|A^T{\bf t}|_2^2/2\Big)\le \exp\Big(\|A\|^2|{\bf t}|_2^2/2\Big).
$$
It means that in this case we have $\tau(Ag)\le\|A\|$. 
\end{exa}
\begin{rem}
For ${\bf t} \in \mathbb{R}^n$ and a sub-gaussian random vector $\xi\in Sub^n(\Omega)$ the random variable 
$\left\langle \xi,{\bf t}\right\rangle\in Sub(\Omega)$. Thus
$$
\mathbb{E}\exp(\left\langle \xi,{\bf t}\right\rangle)\le \exp\big(\tau(\left\langle \xi,{\bf t}\right\rangle)^2/2\big).
$$
By equivalence of norms $\tau(\cdot)$ and $\|\cdot\|_{\psi_2}$ we get
\begin{eqnarray*}
\exp\big(\tau(\left\langle \xi,{\bf t}\right\rangle)^2/2\big)&\le & \exp(C_2^2\|\left\langle \xi,{\bf t}\right\rangle\|_{\psi_2}^2/2\Big)=
\exp\Big(C_2^2\|\left\langle \xi,{\bf t}/|{\bf t}|_2\right\rangle\|_{\psi_2}^2|{\bf t}|_2^2/2\Big)\\
\; & \le & \exp\Big(C_2^2\|\xi\|_{\psi_2}^2|{\bf t}|_2^2/2\Big),\\
\end{eqnarray*}
where $C_2$ is the same as in one-dimensional case. The obtained estimate 
$\mathbb{E}\exp(\left\langle \xi,{\bf t}\right\rangle)\le \exp\Big(C_2^2\|\xi\|_{\psi_2}^2|{\bf t}|_2^2/2\Big)$ means that
$\tau(\xi)\le C_2\|\xi\|_{\psi_2}$ also in the multi-dimensional case.

\end{rem}

\section{Results}

To prove our main results we will need some technical lemma.  
\begin{lem}
\label{lem2}
Let $\eta$ be a centered random variable. If there exist positive constants  $a,b$ such that 
$\mathbb{E}\exp(t\eta)\le \exp(a^2t^2/2)$ for $|t|\le b$ 
then for  every $s\ge 0$ we have
$$
\mathbb{P}\big(|\eta|\ge s\big)\le 2e^{-g(s)},
$$
where 
$$
g(s)=\left\{
\begin{array}{ccl}
\frac{s^2}{2a^2} & {\rm if} & 0\le s \leq a^2b,\\
bs-\frac{a^2b^2}{2} & {\rm if} &   a^2b<s.
\end{array}
\right.
$$
\end{lem}
\begin{proof}
Define  a function 
$$
\varphi_{a,b}(t)=\left\{
\begin{array}{ccl}
\frac{a^2t^2}{2} & {\rm if} & |t|\le b,\\
\infty & {\rm otherwise} & .
\end{array}
\right.
$$
Observe that $\ln\mathbb{E}\exp(t\eta)\le \varphi_{a,b}(t)$ for every $t\in\mathbb{R}$. One can check that the convex conjugate of $\varphi_{a,b}$, i.e.  
$\varphi_{a,b}^\ast(s)=\sup_{t\in\mathbb{R}}\{ts-\varphi_{a,b}(t)\}$, equals 
$$
\varphi_{a,b}^\ast(s)=\left\{
\begin{array}{ccl}
\frac{s^2}{2a^2} & {\rm if} & |s|\le a^2b,\\
b s-\frac{a^2b^2}{2} & {\rm if} & |s| > a^2b.
\end{array}
\right.
$$
By exponential Markov's inequality and the  estimate on the cumulant generating function $\ln\mathbb{E}e^{t\eta}$, the inequality 
$$
\mathbb{P}\big(\eta\ge s\big)\le e^{-st}e^{\ln\mathbb{E}e^{t\eta}}\le e^{-\{st-\varphi_{a,b}(t)\}},
$$
holds for $s,t>0$, giving 
$$
\mathbb{P}\big(\eta\ge s\big)\le \inf_{t>0} e^{-\{st-\varphi_{a,b}(t)\}}=e^{-\sup_{t>0}\{st-\varphi_{a,b}(t)\}}=e^{-\varphi_{a,b}^\ast(s)},
$$
since $\sup_{t>0}\{st-\varphi_{a,b}(t)\}=\varphi_{a,b}^\ast(s)$ for the even function $\varphi_{a,b}$. The inequality 
$\mathbb{P}(\eta\le -s)\le \exp(-\varphi_{a,b}^\ast(s))$ is proved similarly. Combining ones and taking $g=\varphi_{a,b}^\ast$ we get the proof.

\end{proof}
 
\begin{rem}
\label{remlem2}
Let us observe that $g(t)\ge \min\{t^2/(2a^2),bt/2\}$ and we may rewrite the claim of the above lemma in a weaker but more traditional, for the Bernstein-type inequality, form  as follows 
$$
\mathbb{P}\big(|\eta|\ge t\big)\le 2\exp\Big(-\min\Big\{\frac{t^2}{2a^2},\frac{bt}{2}\Big\}\Big).
$$
\end{rem}
\begin{rem}
\label{rem2lem2}
By virtue of (\ref{estMgf}) we know that centered sub-exponential random variables satisfy the assumption of the above lemma with $a=\sqrt{2}C_1\|\eta\|_{\psi_1}$ and $b=1/(C_1\|\eta\|_{\psi_1})$. 
So, for such variables, we can rewrite the above estimate in the following way
$$
\mathbb{P}\big(|\eta|\ge t\big)\le 2\exp\Big(-\min\Big\{\frac{t^2}{4C_1^2\|\eta\|_{\psi_1}^2},\frac{t}{2C_1\|\eta\|_{\psi_1}}\Big\}\Big).
$$
Let us emphasize that one of the ways to obtain Bernstein-type inequalities for sub-exponential random variables is to find (estimate) of their $\psi_1$-norms.
\end{rem}

We show that, for $\xi\in L^n_{\psi_2}(\Omega)$ and $n\times n$ matrix $A$, a random variable $\xi^TA\xi-\mathbb{E}(\xi^TA\xi)$ is centered and sub-exponential.
We estimate its $\psi_1$-norm and, in this way, we get Bernstein-type inequality.

We will use the standard inner product notation to write quadratic form, i.e. $\xi^TA\xi=\left\langle A\xi, \xi \right\rangle$. To obtain the Bernstein-type estimations, it is enough to estimate the $\psi_1$-norms of centered quadratic forms $\left\langle A\xi, \xi \right\rangle$. 
We will also need some estimate for the $\psi_1$-norm of $\mathbb{E}\left\langle A\xi, \xi \right\rangle$ by the $\psi_1$-norm of the quadratic forms $\left\langle A\xi, \xi \right\rangle$.
By the definition of the Luxemburg norm
and the Jensen inequality applied to a convex function $\exp\{|\cdot|/a\}$ ($a>0$) we get
\begin{equation*}
\label{estE}
2\ge \mathbb{E}\exp\Big(\frac{|\left\langle A\xi, \xi \right\rangle|}{\|\left\langle A\xi, \xi \right\rangle\|_{\psi_1}}\Big)\ge \exp\Big(\frac{|\mathbb{E}\left\langle A\xi, \xi \right\rangle|}{\|\left\langle A\xi, \xi \right\rangle\|_{\psi_1}}\Big)=
\mathbb{E}\exp\Big(\frac{|\mathbb{E}\left\langle A\xi, \xi \right\rangle|}{\|\left\langle A\xi, \xi \right\rangle\|_{\psi_1}}\Big),
\end{equation*}
which means that $\|\mathbb{E}\left\langle A\xi, \xi \right\rangle\|_{\psi_1}\le\|\left\langle A\xi, \xi \right\rangle\|_{\psi_1}$.
It follows that 
\begin{equation}
\label{estcent}
\|\left\langle A\xi, \xi \right\rangle-\mathbb{E}\left\langle A\xi, \xi \right\rangle\|_{\psi_1}\le 2\|\left\langle A\xi, \xi \right\rangle\|_{\psi_1}. 
\end{equation}

Now we formulate and prove our main  results.
\begin{pro}
\label{lemtr}
Let $\xi\in  L^n_{\psi_2}(\Omega)$ and $A$ be $n\times n$ matrix. Then $\left\langle A\xi,\xi\right\rangle\in L_{\psi_1}(\Omega)$ and
$$
\|\left\langle A\xi,\xi\right\rangle\|_{\psi_1}\le \|A\|_{tr}\|\xi\|_{\psi_2}^2,
$$
where $\|A\|_{tr}=trace(AA^T)^{1/2}$ is the trace norm of $A$. 
\end{pro}
\begin{proof}
Since $\left\langle A\xi,\xi\right\rangle=\left\langle 1/2(A+A^T)\xi,\xi\right\rangle$, we may assume that $A$ is symmetric matrix. Moreover, because we can  present it as the difference of two symmetric and nonnegative definite matrices $A_1$ and $A_2$ such that $\|A\|_{tr}=\|A_1-A_2\|_{tr}=\|A_1\|_{tr}+\|A_2\|_{tr}$, then
$$
\|\left\langle A\xi,\xi\right\rangle\|_{\psi_1}\le \|A_1\|_{tr}\|\xi\|_{\psi_2}^2+\|A_2\|_{tr}\|\xi\|_{\psi_2}^2=\|A\|_{tr}\|\xi\|_{\psi_2}^2
$$ 
and, without loss of generality, we may also assume that this matrix is nonnegative definite.  

Let $US U^T$ be the singular-value decomposition (eigendecomposition) of symmetric and nonnegative definite matrix $A$,
where $U$ is a unitary matrix and  $S=diag(s_1,...,s_n)$ is a diagonal matrix with singular values (eigenvalues) $s_i$ of $A$ on the diagonal. This allows us to describe  
$\left\langle A\xi,\xi\right\rangle$ as $|S^{1/2}U^T\xi|^2_2=\sum_{i=1}^ns_i\left\langle \xi,Ue_i\right\rangle^2$. Thus 
\begin{equation*}
\mathbb{E}\exp\Big(\frac{\left\langle A\xi,\xi\right\rangle}{\|A\|_{tr}\|\xi\|_{\psi_2}^2}\Big)=
\mathbb{E}\exp\Big(\frac{\sum_{i=1}^ns_i\left\langle \xi,Ue_i\right\rangle^2}{\|A\|_{tr}\|\xi\|_{\psi_2}^2}\Big)=
\mathbb{E}\Big(\prod_{i=1}^n\exp\frac{ s_i\left\langle \xi,Ue_i\right\rangle^2}{\|A\|_{tr}\|\xi\|_{\psi_2}^2}\Big).
\end{equation*} 
Using the multi-factorial H\"older inequality with exponents $p_i=\sum_{k=1}^ns_k/s_i$, $1\le i\le n$, we get
\begin{eqnarray*}
\mathbb{E}\Big(\prod_{i=1}^n\exp\frac{ s_i\left\langle \xi,Ue_i\right\rangle^2}{\|A\|_{tr}\|\xi\|_{\psi_2}^2}\Big)&\le& 
\prod_{i=1}^n\Big(\mathbb{E}\exp\Big(\frac{(\sum_{k=1}^ns_k)\left\langle \xi,Ue_i\right\rangle^2}{\|A\|_{tr}\|\xi\|_{\psi_2}^2}\Big)\Big)^{\frac{s_i}{\sum_{k=1}^ns_k}}\\
\; &=& \prod_{i=1}^n\Big(\mathbb{E}\exp\Big(\frac{\left\langle \xi,Ue_i\right\rangle^2}{\|\xi\|_{\psi_2}^2}\Big)\Big)^{\frac{s_i}{\sum_{k=1}^ns_k}},
\end{eqnarray*}
since $\|A\|_{tr}=\sum_{k=1}^ns_k$. Because $|Ue_i|_2=1$ then we can estimate the right hand side as follows 
$$
\prod_{i=1}^n\Big(\mathbb{E}\exp\Big(\frac{\left\langle \xi,Ue_i\right\rangle^2}{\|\xi\|_{\psi_2}^2}\Big)\Big)^{\frac{s_i}{\sum_{k=1}^ns_k}}\le
\prod_{i=1}^n\Big(\sup_{|{\bf t}|_2=1}\mathbb{E}\exp\Big(\frac{\left\langle \xi,{\bf t}\right\rangle^2}{\|\xi\|_{\psi_2}^2}\Big)\Big)^{\frac{s_i}{\sum_{k=1}^ns_k}}\le 2,
$$
since each factor is less than or equal to $2^{s_i/\sum_{k=1}^ns_k}$, $i=1,...,n$.

In summary, we get  
$$
\mathbb{E}\exp\Big(\frac{\left\langle A\xi,\xi\right\rangle}{\|A\|_{tr}\|\xi\|_{\psi_2}^2}\Big)\le 2,
$$
which means that $\|\left\langle A\xi,\xi\right\rangle\|_{\psi_1}\le \|A\|_{tr}\|\xi\|_{\psi_2}^2$.
\end{proof}
Immediately by condition (\ref{estcent}), 
Proposition \ref{lemtr} and Remark \ref{rem2lem2} we get our first tail estimate for quadratic forms in sub-gaussian dependent random variables.
\begin{cor}
Let $\xi\in  L^n_{\psi_2}(\Omega)$ and $\|\xi\|_{\psi_2}\le K$. Then, for $n\times n$ matrix $A$, we have
$$
\mathbb{P}\Big(\big|\left\langle A\xi,\xi\right\rangle-\mathbb{E}\left\langle A\xi,\xi\right\rangle\big|\ge t\Big)\le 
2\exp\Big(-\min\Big\{\frac{t^2}{16C_1^2\|A\|_{tr}^2K^4},\frac{t}{4C_1\|A\|_{tr}K^2}\Big\}\Big).
$$
\end{cor}

If we additionally assume that $\mathbb{E}\xi={\bf 0}$, then we can get a better estimate with the Hilbert-Schmidt norm instead of the trace norm of $A$.
\begin{pro}
\label{lemHS}
Let $\xi\in  Sub^n(\Omega)$ and $A$ be $n\times n$ symmetric and nonnegative definite matrix. Then 
$$
\|\left\langle A\xi,\xi\right\rangle\|_{\psi_1}\le 2C_2\|A\|_{HS}\|\xi\|_{\psi_2}^2.
$$
\end{pro}
\begin{proof}
We start with some form of the decoupling argument.
Consider two independent Bernoulli random variables $\delta_k\in\{0,1\}$ with $\mathbb{E}\delta_k=1/2$, $k=1,2$, which are also independent from $\xi$. Define new random vectors 
$\delta_k\xi=(\delta_k\xi_i)_{i=1}^n$. 
Note that these vectors are centered, i.e. $\mathbb{E}_{\xi,\delta_k}(\delta_k\xi)={\bf 0}\in\mathbb{R}^n$, where $\mathbb{E}_{\xi,\delta_k}$ denote expectation with respect to both $\xi$ and $\delta_k$. Notice that
\begin{equation}
\label{est1}
\mathbb{E}_{\xi,\delta_k}e^{\left\langle {\bf t},\delta_k\xi\right\rangle}=
\mathbb{E}_{\xi,\delta_k}e^{\delta_k\left\langle {\bf t},\xi\right\rangle}=\mathbb{E}\Big(e^{\left\langle {\bf t},\xi\right\rangle}/2+1/2\Big)
\le e^{\tau(\xi)^2|{\bf t}|_2^2/2}.
\end{equation}
 It means that, for $k=1,2$, $\tau(\delta_k\xi)\le \tau(\xi)$.
By $\mathbb{E}_{\delta}$  we will denote expectation with respect to  $\delta=(\delta_1,\delta_2)$ and by $\mathbb{E}_{\xi,\delta}$ with respect to both $\xi$ and $\delta$. Since $\mathbb{E}_\delta(\delta_1\delta_2)=1/4$, we have  
$
\left\langle A\xi,\xi\right\rangle=4\mathbb{E}_{\delta}\left\langle A(\delta_1\xi),\delta_2\xi\right\rangle.
$  
Jensen's inequality yields
$$
\mathbb{E}e^{\left\langle A\xi,\xi\right\rangle}\le \mathbb{E}_{\xi,\delta}e^{4\left\langle A(\delta_1\xi),\delta_2\xi\right\rangle}.
$$
Conditioning with respect $\delta_1\xi$ and using the definition of the norm $\tau(\cdot)$ we get
\begin{eqnarray*}
\mathbb{E}_{\xi,\delta}e^{4\left\langle A(\delta_1\xi),\delta_2\xi\right\rangle} &=&
\mathbb{E}_{\xi,\delta_1}\Big[\mathbb{E}_{\xi,\delta_2}\Big(e^{4\left\langle A(\delta_1\xi),\delta_2\xi\right\rangle}\Big|\delta_1\xi\Big)\Big]\\
\; &\le& 
\mathbb{E}_{\xi,\delta_1}e^{2 |A(\delta_1\xi)|_2^2\tau(\delta_2\xi)^2}\le \mathbb{E}_{\xi,\delta_1}e^{2 |A(\delta_1\xi)|_2^2\tau(\xi)^2}.
\end{eqnarray*}
The second inequality follows from $\tau(\delta_2\xi)\le\tau(\xi)$. Similarly as in (\ref{est1}) one can show that the above right hand side is less than or equal to $\mathbb{E}\exp(2 |A\xi|_2^2\tau(\xi)^2)$. 
Summing up we get the following
$$
\mathbb{E}\exp\Big(\frac{\left\langle A\xi,\xi\right\rangle}{2\|A\|_{HS}\|\xi\|_{\psi_2}\tau(\xi)}\Big)  \le
\mathbb{E}\exp\Big(\frac{2 |A\xi|_2^2\tau(\xi)^2}{2\|A\|_{HS}^2\|\xi\|_{\psi_2}^2\tau(\xi)^2}\Big)=
\mathbb{E}\exp\Big(\frac{ |A\xi|_2^2}{\|A\|_{HS}^2\|\xi\|_{\psi_2}^2}\Big).
$$
Let $USU^T$ be again the eigendecomposition 
of $A$, i.e. $U$ is some unitary matrix 
and  $S=diag(s_1,...,s_n)$ is a diagonal matrix with eigenvalues $s_i$ of $A$ on the diagonal. Since $U$ is the unitary matrix, we have  $|A\xi|_2^2=|S U^T\xi|_2^2$. It can be rewritten as 
$$
|SU^T\xi|_2^2=
\sum_{i=1}^n\left\langle SU^T\xi,e_i\right\rangle^2=\sum_{i=1}^ns_i^2\left\langle \xi,Ue_i\right\rangle^2.
$$
Thus 
\begin{equation*}
\mathbb{E}\exp\Big(\frac{\left\langle A\xi,\xi\right\rangle}{2\|A\|_{HS}\|\xi\|_{\psi_2}\|\xi\|_{S_{(2)}}}\Big)\le
\mathbb{E}\exp\Big(\frac{\sum_{i=1}^ns^2_i\left\langle \xi,Ue_i\right\rangle^2}{\|A\|_{HS}^2\|\xi\|_{\psi_2}^2}\Big)=
\mathbb{E}\Big(\prod_{i=1}^n\exp\frac{ s_i^2\left\langle \xi,Ue_i\right\rangle^2}{\|A\|_{HS}^2\|\xi\|_{\psi_2}^2}\Big).
\end{equation*} 
Recall that the Hilbert--Schmidt norm of $A$ equals $(\sum_{i=1}^n s_i^2)^{1/2}$.
Using the multi-factorial H\"older inequality with exponents $p_i=\sum_{k=1}^ns_k^2/s_i^2$, $1\le i\le n$, we get
\begin{eqnarray*}
\mathbb{E}\Big(\prod_{i=1}^n\exp\frac{ s^2_i\left\langle \xi,Ue_i\right\rangle^2}{\|A\|_{HS}^2\|\xi\|_{\psi_2}^2}\Big)&\le& 
\prod_{i=1}^n\Big(\mathbb{E}\exp\Big(\frac{(\sum_{k=1}^ns^2_k)\left\langle \xi,Ue_i\right\rangle^2}{\|A\|_{HS}^2\|\xi\|_{\psi_2}^2}\Big)\Big)^{\frac{s^2_i}{\sum_{k=1}^ns^2_k}}\\
\; &=& \prod_{i=1}^n\Big(\mathbb{E}\exp\Big(\frac{\left\langle \xi,Ue_i\right\rangle^2}{\|\xi\|_{\psi_2}^2}\Big)\Big)^{\frac{s^2_i}{\sum_{k=1}^ns^2_k}}.
\end{eqnarray*}
Because $|Ue_i|_2=1$ then we can estimate the right hand side as follows 
$$
\prod_{i=1}^n\Big(\mathbb{E}\exp\Big(\frac{\left\langle \xi,Ue_i\right\rangle^2}{\|\xi\|_{\psi_2}^2}\Big)\Big)^{\frac{s^2_i}{\sum_{k=1}^ns^2_k}}\le
\prod_{i=1}^n\Big(\sup_{|{\bf t}|_2=1}\mathbb{E}\exp\Big(\frac{\left\langle \xi,{\bf t}\right\rangle^2}{\|\xi\|_{\psi_2}^2}\Big)\Big)^{\frac{s^2_i}{\sum_{k=1}^ns^2_k}}\le 2,
$$
since each factor is less than or equal to $2^{s^2_i/\sum_{k=1}^ns^2_k}$, $i=1,...,n$.

Summarizing we get 
$$
\mathbb{E}\exp\Big(\frac{\left\langle A\xi,\xi\right\rangle}{2\|A\|_{HS}\|\xi\|_{\psi_2}\tau(\xi)}\Big)\le 2,
$$
which means that $\|\left\langle A\xi,\xi\right\rangle\|_{\psi_1}\le 2\|A\|_{HS}\|\xi\|_{\psi_2}\tau(\xi)\le 2C_2\|A\|_{HS}\|\xi\|_{\psi_2}^2$.
\end{proof}
\begin{rem}
\label{estHS}
Considering quadratic forms we can always assume that their generating  matrices are symmetric. But now for the Hilbert--Schmidt norm of $A=A_1-A_2$ ($A_1$ and $A_2$ are symmetric and nonnegative definite) we only have $\|A_1\|_{HS}+\|A_2\|_{HS}\le\sqrt{2}\|A_1-A_2\|_{HS}=\sqrt{2}\|A\|_{HS}$. Thus for  arbitrary $A$ we  get
$$
\|\left\langle A\xi,\xi\right\rangle\|_{\psi_1}\le 2\sqrt{2}\|A\|_{HS}\|\xi\|_{\psi_2}\tau(\xi)\le 2\sqrt{2}C_2\|A\|_{HS}\|\xi\|_{\psi_2}^2. 
$$
\end{rem}
By virtue of Remark \ref{rem2lem2},  
Proposition \ref{lemHS} and Remark \ref{estHS} we can formulate the following
\begin{cor}
\label{twHS}
Let $\xi\in  Sub^n(\Omega)$ and $\|\xi\|_{\psi_2}\le K$. Then, for $n\times n$ matrix $A$, we have
$$
\mathbb{P}\Big(\big|\left\langle A\xi,\xi\right\rangle-\mathbb{E}\left\langle A\xi,\xi\right\rangle\big|\ge t\Big)\le 
2\exp\Big(-\min\Big\{\frac{t^2}{C_3^2\|A\|_{HS}^2K^4},\frac{t}{C_3\|A\|_{HS}K^2}\Big\}\Big),
$$
where $C_3=2\sqrt{2}C_1C_2$.
\end{cor}
\begin{rem}
The above estimate has the form similar to the Hanson--Wright inequality but with one important difference. In the second term of the minimum in  the Hanson--Wright inequality there is the operator norm of $A$. It is an important (better) component, especially in the context of the long-time behavior of investigated estimates.
\end{rem}
\begin{rem}
These type inequalities can be obtained under more general assumption that a random vector $\xi$ satisfy the convex concentration property (see \cite{VuW,Ada} for instance). Adamczak in \cite{Ada} get the form of the Hanson-Wright inequality under assumption that 
a random vector $\xi$ satisfy the convex concentration property with constant $K$ (see \cite[Th.2.5]{Ada}). 
As the author notes this constant depends on the dimension $n$, even for vectors with independent sub-gaussian coordinates (see \cite[Rem. 2.9]{Ada}).
\end{rem}
Let us emphasize that our results have been obtained without the additionally assumption of the convex concentration  of $\xi$. Moreover our $K$ 
is simply the norm of $\xi$ and not a slightly enigmatic constant for which the convex concentration property holds.

In the paper Hsu et al. \cite{HKZ} one can find some application of the our norm $\tau$ (without using this notion and notation) to the investigation of the quadratic form 
of the form 
$|A\xi|_2^2$ (see \cite[Th.2.1]{HKZ}). The authors consider only estimation of probability for one side tails  of these forms which are sufficient for their applications.

For the sake of completeness, we will present in the next remark a  derivation of the Hanson--Wright inequality for independent normally distributed random variables.
\begin{rem}
Decoupling argument applied to $\xi=(\xi_i)\in Sub^n(\Omega)$ with independent coordinates $\xi_i$ implies that we can investigate 
$\left\langle Ag,g\right\rangle$ ($g\sim\mathcal{N}({\bf 0},I_n)$) instead of $\left\langle A\xi,\xi\right\rangle$ (up to a product of some universal constant $c$ and  $\psi_2$-norm of $\xi$). 

Recall that if $g\sim\mathcal{N}(0,1)$ then $g^2$ has $\chi^2_1$-distribution with one degree of freedom, whose
 moment generating function is $\mathbb{E}\exp(tg^2)=(1-2t)^{-1/2}$ for $t<1/2$.
 Therefore 
$$
\mathbb{E}\exp\Big(\frac{g^2}{K^2}\Big)=\Big(1-\frac{2}{K^2}\Big)^{-1/2},
$$
which is less than or equal to $2$ if $K\ge \sqrt{8/3}$.
It implies that $\|g\|_{\psi_2}=\sqrt{8/3}$. 
Because for any ${\bf t}\in S^{n-1}$ and $g\sim\mathcal{N}({\bf 0},I_n)$ the inner product $\left\langle {\bf t},g\right\rangle\sim \mathcal{N}(0,1)$ then by the above and the definition of $\psi_2$-norm of random vectors we also have that  $\|g\|_{\psi_2}=\sqrt{8/3}$ for $g\sim\mathcal{N}({\bf 0},I_n)$.

Let $A$ be $n\times n$ symmetric nonnegative definite matrix and
$USU^T$ be its singular value decomposition; i.e. $A=USU^T$, $S=diag(s_1,...,s_n)$ and $U^{-1}=U^T$. Then for $g\sim\mathcal{N}({\bf 0},I_n)$ we get
$$
\left\langle Ag,g\right\rangle-\mathbb{E}\left\langle Ag,g\right\rangle
=\sum_{i=1}^ns_i\Big(\left\langle g,Ue_i\right\rangle^2-\mathbb{E}\left\langle g,Ue_i\right\rangle^2\Big)
$$
and 
$$
\mathbb{E}\exp\Big(t(\left\langle Ag,g\right\rangle-\mathbb{E}\left\langle Ag,g\right\rangle)\Big) =  
\mathbb{E}\exp\Big(t\sum_{i=1}^ns_i\Big(\left\langle g,Ue_i\right\rangle^2-\mathbb{E}\left\langle g,Ue_i\right\rangle^2\Big)\Big).
$$
Rotational invariance of Gaussian distribution implies that $\left\langle g,Ue_i\right\rangle$, $i=1,...,n$, are independent random variables. In consequence,
\begin{equation}
\label{prod}
\mathbb{E}\exp\Big(t(\left\langle Ag,g\right\rangle-\mathbb{E}\left\langle Ag,g\right\rangle)\Big)=
\prod_{i=1}^n\mathbb{E}\exp\Big(ts_i\Big(\left\langle g,Ue_i\right\rangle^2-\mathbb{E}\left\langle g,Ue_i\right\rangle^2\Big)\Big).
\end{equation}
Moreover, the $\psi_2$-norm of each  $\left\langle g,Ue_i\right\rangle$, $i=1,...,n$, is equal to  the $\psi_2$-norm of whole vector $g$, i.e.
$$
\|\left\langle g,Ue_i\right\rangle\|_{\psi_2}=\|g\|_{\psi_2}=\sqrt{8/3}.
$$
and 
$$
\|\left\langle g,Ue_i\right\rangle^2\|_{\psi_1}=\|\left\langle g,Ue_i\right\rangle\|_{\psi_2}^2= \|g\|_{\psi_2}^2=8/3.
$$
By the above and (\ref{estMgf}), for each factor of (\ref{prod}), we get
$$
\mathbb{E}\exp\Big(ts_i\Big(\left\langle g,Ue_i\right\rangle^2-\mathbb{E}\left\langle g,Ue_i\right\rangle^2\Big)\Big)
\le \exp(256C_1^2s_i^2t^2/9),
$$
if $|t|\le 3/(16C_1s_i)$. Substituting of the above estimate into (\ref{prod}), taking into account that $\sum_{i=1}^ns_i^2=\|A\|_{HS}$ and 
$\max_is_i=\|A\|$, we obtain the following 
$$
\mathbb{E}\exp\Big(t(\left\langle Ag,g\right\rangle-\mathbb{E}\left\langle Ag,g\right\rangle)\Big)\le \exp\Big(256C_1^2\|A\|_{HS}^2t^2/9\Big)
$$
for $|t|\le 3/(16C_1\|A\|)$.

We have proved that for $\eta=\left\langle Ag,g\right\rangle-\mathbb{E}\left\langle Ag,g\right\rangle$ we can take in Lemma \ref{lem2} 
$a=16\sqrt{2}C_1\|A\|_{HS}/3$ and $b=3/(16C_1\|A\|)$.
By Remark \ref{remlem2} we obtain the form of Hanson--Wright's inequality for Gaussian random vector $g\sim\mathcal{N}({\bf 0},I_n)$:
$$
\mathbb{P}\Big(\big|\left\langle Ag,g\right\rangle-\mathbb{E}\left\langle Ag,g\right\rangle\big|\ge t\Big)\le 
2\exp\Big(-\min\Big\{\frac{9t^2}{512C_1^2\|A\|_{HS}^2},\frac{3t}{32C_1\|A\|}\Big\}\Big).
$$
\end{rem}

We show an application of Corollary \ref{twHS} to fixed design linear regression (compare  Hsu et al. \cite{HKZ}). 
\begin{exa}
Let ${\bf x}_1,...,{\bf x}_n\in \mathbb{R}^d$ and $X=({\bf x}_1,...,{\bf x}_n)$ denote the $d\times n$  design matrix. Assume that 
$\Sigma=n^{-1}\sum_{i=1}^n{\bf x}_i{\bf x}_i^T$ is invertible. 
Let $\xi=(\xi_i)_{i=1}^n\in L_{\psi_2}^n$. Define the coefficient vector of the least expected squared error (given the observation $\xi$):
$$
\beta:=n^{-1}\sum_{i=1}^n\mathbb{E}\xi_i\Sigma^{-1}{\bf x}_i,
$$
and its ordinary least squares estimator:
$$
\hat{\beta}(\xi):=n^{-1}\sum_{i=1}^n\xi_i\Sigma^{-1}{\bf x}_i.
$$

The quality of the estimator $\hat{\beta}$ can be judged by the excess loss
$$
R(\xi)=|\Sigma^{1/2}(\hat{\beta}(\xi)-\beta)|_2^2=\left\langle A(\xi-\mathbb{E}\xi),\xi-\mathbb{E}\xi\right\rangle,
$$
where $A=n^{-2}X^T\Sigma^{-1}X$ as can be shown by algebraic calculations.
Suppose $\|\xi\|_{\psi_2}\le K$ (recall that $\|\xi-\mathbb{E}\xi\|_{\psi_2}\le 2\|\xi\|_{\psi_2}$). By an equivalent form of Corollary \ref{twHS} we get the following tail estimate
$$
\mathbb{P}\Big(\big|R(\xi)-\mathbb{E}R(\xi)\big|\ge \|A\|_{HS}K^2\max\big\{\sqrt{t},t\big\}\Big)\le 2\exp(-t/C_4),
$$
where $C_4=2C_3=4\sqrt{2}C_1C_2$.
\end{exa}
\begin{rem}
In the paper \cite{Zaj} one can find another estimate of quadratics forms (even chaoses of higher order) by using other norms of random vectors and, in consequence, new forms of their tail estimates (see \cite[Rem. 3.6]{Zaj}).
\end{rem}

\end{document}